\documentclass[10pt]{amsart}
\usepackage[utf8]{inputenc}

\usepackage{amssymb,amsthm,amsmath}
\usepackage{mathrsfs}
\usepackage{enumerate}
\usepackage{graphicx,xcolor}
\usepackage[hidelinks]{hyperref}
\usepackage{setspace}

\usepackage[paper=a4paper, left=1.2in, right=1.2in, top=1.2in, bottom=1.2in]{geometry}
\newcommand{\dd}{\mathrm{d}}
\newcommand{\E}{\mathbb{E}}

\newcommand{\1}{\textbf{1}}
\newcommand{\R}{\mathbb{R}}
\newcommand{\C}{\mathbb{C}} 

\newcommand{\e}{\varepsilon}
\newcommand{\cM}{\mathcal{M}}
\newcommand{\cU}{\mathcal{U}}
\newcommand{\cW}{\mathcal{W}}
\newcommand{\cB}{\mathcal{B}}
\newcommand{\cF}{\mathcal{F}}
\newcommand{\cG}{\mathcal{G}}
\newcommand{\p}[1]{\mathbb{P}\left( #1 \right)}

\newcommand{\n}[1]{\left\| #1 \right\|}
\newcommand{\red}{}
\newcommand{\purp}{}

\DeclareMathOperator{\sgn}{sgn}

\newtheorem{theorem}{Theorem}
\newtheorem{lemma}[theorem]{Lemma}
\newtheorem{corollary}[theorem]{Corollary}

\theoremstyle{remark}
\newtheorem{remark}[theorem]{Remark}

\theoremstyle{definition}

\title{Sharp Rosenthal-type inequalities for mixtures\\ and log-concave variables}

\author{Giorgos Chasapis}

\author{Alexandros Eskenazis}

\author{Tomasz Tkocz}

\address{(G.~C.) Carnegie Mellon University; Pittsburgh, PA 15213, USA.\newline Current affiliation: University of Crete, Voutes Campus 70013, Heraklion, Crete, Greece.}
\email{gchasapis@uoc.gr}

\address{(A.~E.) CNRS, Institut de Math\'ematiques de Jussieu, Sorbonne Universit\'e, France and Trinity College, University of Cambridge, UK.}
\email{alexandros.eskenazis@imj-prg.fr, ae466@cam.ac.uk}

\address{(T.~T.) Carnegie Mellon University; Pittsburgh, PA 15213, USA.}
\email{ttkocz@andrew.cmu.edu}

\thanks{G.C. was supported by the Hellenic Foundation for Research and Innovation, Project HFRI-FM17-1733 and by University of Crete Grant 4725. AE was partially supported by a Junior Research Fellowship from Trinity College, Cambridge. TT's research supported in part by NSF grant DMS-1955175.}

\begin{document}

\begin{abstract} 
We obtain Rosenthal-type inequalities with sharp constants for moments of sums of independent random variables which are mixtures of a fixed distribution. We also identify extremisers in log-concave settings when the moments of summands are individually constrained.
\end{abstract}

\maketitle

\bigskip

\begin{footnotesize}

\end{footnotesize}

\bigskip


\section{Introduction}

Rosenthal's inequality, discovered in \cite{R} in connection with questions in the geometry of Banach spaces, provides matching lower and upper bounds for $p$-norms of sums of independent symmetric random variables in terms of norms of the individual summands, at the expense of a multiplicative constant depending only on $p$.  Recall that a random variable $X$ is symmetric if it has the same distribution as $-X$ or, equivalently, if it has the same distribution as $\e X$, where $\e$ is a {\red Rademacher random variable (a symmetric random sign, $\p{\e = \pm 1} = \frac12$)} independent of $X$. Fix $p > 2$ and let $X_1, X_2,\dots$ be independent symmetric random variables in $L_p$. Plainly, for any $n\in\mathbb{N}$ we have
\[
\n{\sum_{j=1}^n X_j}_p \geq \n{\sum_{j=1}^n X_j}_2 = \left(\sum_{j=1}^n \n{X_j}_2^2\right)^{1/2},
\]
where here and throughout $\|Y\|_p = (\E|Y|^p)^{1/p}$ is the $p$-norm of a random variable $Y$.  Moreover, for independent Rademacher random variables $\e_1, \e_2, \dots$ independent of $X_1,X_2,\ldots$,  we have
\begin{align*}
\n{\sum_{j=1}^n X_j}_p^p = \E_X\E_{\e}\left|\sum_{j=1}^n \e_jX_j\right|^p \geq \E_X\left(\E_{\e}\left|\sum_{j=1}^n \e_jX_j\right|^2\right)^{p/2} = \E_X\left(\sum_{j=1}^n |X_j|^2\right)^{p/2} 
\geq \sum_{j=1}^n \E|X_j|^p,
\end{align*}
thus
\begin{equation}\label{eq:ros-easy}
\n{\sum_{j=1}^n X_j}_p \geq \max\left\{\left(\sum_{j=1}^n \n{X_j}_2^2\right)^{1/2}, \left(\sum_{j=1}^n \n{X_j}_p^p\right)^{1/p}\right\}
\end{equation}
and the multiplicative constant $1$ in front of the maximum in this inequality is clearly optimal (simply consider $n=1$). Rosenthal in his influential paper \cite{R} established a reversal of \eqref{eq:ros-easy}: 
for every $2 < p < \infty$, there is a constant $C_p$ which depends only on $p$ such that for every $n\in\mathbb{N}$ and every independent symmetric random variables  $X_1, \dots, X_n$  in $L_p$, we have
\begin{equation}\label{eq:ros}
\n{\sum_{j=1}^n X_j}_p \leq C_p\max\left\{\left(\sum_{j=1}^n \n{X_j}_2^2\right)^{1/2}, \left(\sum_{j=1}^n \n{X_j}_p^p\right)^{1/p}\right\}.
\end{equation}

Rosenthal's original motivation in \cite{R} was to construct new subspaces of $L_p(\mu)$ spaces, answering important questions raised in \cite{LP} and \cite{LR}. His fundamental inequality \eqref{eq:ros}, which substantially generalises Khinchin's inequality, has also received significant attention in probabilistic literature, prompting several fruitful lines of research. To mention only some in passing, there are generalisations to dependent settings of martingale differences, multilinear forms, and the like (see \cite{Bur, Hit,ISC,PIS, P2, P1}), random vectors in Banach spaces (see \cite{Tal,P1}), noncommutative settings (see \cite{JX1,JPX, JX2, JZ}), as well as many works devoted to optimal inequalities with sharp constants in various setups (see, e.g.  \cite{Pr,CK,PU,JSZ,U,FHJSZ,ISh,Sch,P5,II,Ru,P3,P4}).

To describe the last direction in more detail, let ${\bf C}_p$ denote the best constant in Rosenthal's inequality, namely the least $C_p$ such that \eqref{eq:ros} holds for every $n$ and every sequence of independent symmetric random variables $X_1,X_2,\ldots$ in $L_p$. Rosenthal's proof gives ${\bf C}_p = O(p^p)$. The sharp behaviour ${\bf C}_p = \Theta(p/\log p)$ as $p \to \infty$ was established by Johnson, Schechtman and Zinn in \cite{JSZ}. This can also be deduced from Lata\l a's precise formula for moments of sums in terms of marginal distributions (see  \cite[Corollary 3]{L}). The exact value of ${\bf C}_p$ is known to be
\begin{equation}\label{eq:ros-Cp}
{\bf C}_p = \begin{cases} (1 + \|Z\|_p^p)^{1/p}, & 2 < p \leq 4, \\ \n{\sum_{j=1}^\xi \e_j}_p, & p \geq 4, \end{cases}
\end{equation}
where $Z$ is a standard Gaussian random variable and $\xi$ is a Poisson random variable with parameter $1$, independent of the Rademacher sequence $\e_1,\e_2,\ldots$.
For $p\geq4$, the value of ${\bf C}_p$ was found by Utev in \cite{U} (continuing investigations of Prokhorov from \cite{Pr} and Pinelis and Utev from \cite{PU}). When $2 < p < 4$,  \eqref{eq:ros-Cp} was proven by Ibragimov and Sharakhmetov in \cite{ISh} (building on Utev's approach), and, independently, by Figiel, Hitczenko, Johnson, Schechtman and Zinn in \cite{FHJSZ} (with different methods treating a more general case of Orlicz functionals in place of moments).

The question of identifying the value of ${\bf C}_p$ can be restated as an extremal problem, since
\begin{equation} \label{eq:ros-extremal}
{\bf C}_p = \sup\left\{ \left\|\sum_{j=1}^n X_j\right\|_p: \ n\in\mathbb{N} \mbox{ and } \max\Big\{\sum_{j=1}^n \n{X_j}_2^2, \sum_{j=1}^n \n{X_j}_p^p\Big\} \leq 1 \right\}.
\end{equation}
This reformulation has led to the study of many \emph{Rosenthal-type} extremal problems, in which the supremal value of the $p$-norm of a sum of independent variables is sought for under more refined assumptions, including constraints of the form $\max\big\{\sum_{j=1}^n \n{X_j}_2^2, \lambda\sum_{j=1}^n \n{X_j}_p^p\big\}\leq 1$ for suitable parameters $\lambda>0$, or $\max\{\|X_j\|_2, \lambda_j \|X_j\|_p\}\leq \mu_j$, where $\lambda_j,\mu_j>0$. As we shall see in the next section, (asymptotic) maximisers of such multi-constraint problems can frequently be identified which, in turn,  often leads to optimal constants in various Rosenthal-type inequalities.

The present paper is concerned with the following question: what is the nature of maximisers of Rosenthal-type extremal problems if we a priori assume that the independent variables $X_1,X_2,\ldots$ have additional properties beyond symmetry,  such as unimodality or log-concavity? Can one identify the optimal value of the corresponding optimal constant ${\bf C}_p$ under these assumptions? Building on Utev's approach from \cite{U}, we fully answer these questions in the case of general mixtures and we also characterize the extremising sequences in the log-concave case,  using an effective refinement \mbox{of the one-dimensional localisation principle (see \cite{LS,FG1,FG2}) developed in \cite{ENT}.}


{\purp
\subsection*{Acknowledgements.} We should like to thank the referee for a careful reading of the manuscript and many invaluable suggestions, including Remarks \ref{rem:Mp>4} and \ref{rem:Xpm}. We are also indebted to Olivier Gu\'edon for the discussion on the complex case.
}


\section{Results}


\subsection{General mixtures}

Let $V$ be a symmetric random variable. We say that a random variable $X$ is a $V$-mixture if it has the same distribution as $RV$ for some nonnegative random variable $R$, independent of $V$. In particular, when $V$ is a Rademacher random variable, $V$-mixtures are exactly symmetric random variables and when $V$ is uniform, $V$-mixtures are exactly symmetric unimodal random variables (see, e.g.,~\cite[Lemma 1]{LO95}). When $V$ is standard Gaussian random variable, we refer to a $V$-mixture as a (symmetric) Gaussian mixture (see \cite{AH, ENTm}).

Following Utev \cite{U}, for $p > 2$, an integer $n \geq 1$,  sequences $a = (a_1, \dots, a_n)$,  $b = (b_1, \dots, b_n)$ of positive numbers and positive parameters $A, B$, we define the classes of $n$-tuples $X = (X_1, \dots, X_n)$ of independent $V$-mixtures with constrained moments as follows:
\begin{equation}
\cM_{V}(n, p, a,b) = \Big\{X: \ \mbox{each } X_j \mbox{ is a } V\mbox{-mixture with} \ \|X_j\|_2 \leq a_j \mbox{ and } \|X_j\|_p \leq b_j\Big\}, 
\end{equation}
and
\begin{equation}
\cU_{V}(n, p, A,B) = \left\{X: \ \mbox{each } X_j \mbox{ is a } V\mbox{-mixture and } \sum_{j=1}^n \|X_j\|_2^2 \leq A^2, \ \sum_{j=1}^n \|X_j\|_p^p \leq B^p\right\},
\end{equation}
where we have implicitly assumed that $a_j \leq b_j$ for every $j$ (for feasibility). We shall simply write $\cM$ and $\cU$ to denote $\cM_V$ and $\cU_V$ with $V$ being a Rademacher random variable, that is for the classes of independent symmetric random variables satisfying the constraints. Moreover, {\red $\cU_V'$ will denote the subclass of $\cU_V$} consisting of all the tuples $X = (X_1, \dots, X_n)$ of \emph{identically distributed} $V$-mixtures satisfying the imposed moment constraints.

Our first theorem is the solution of the following Rosenthal-type extremal problem for mixtures.

\begin{theorem}\label{thm:ros-mix}
Fix $A, B > 0$ and let $V$ be a symmetric random variable. For $2 < p < 4$, we have
\begin{equation}\label{eq:ros-mix-p<4}
 \sup\left\{ \E\left|\sum_{j=1}^n X_j\right|^p: \ n\in\mathbb{N} \mbox{ and } X\in \cU_{V}(n, p, A,B) \right\}= B^p + \|Z\|_p^pA^p,
\end{equation}
where $Z$ is a standard Gaussian, provided that $V$ is in $L_{p+\delta}$ for some $\delta > 0$. For $p \geq 4$, we have
\begin{equation}\label{eq:ros-mix-p>4}
 \sup\left\{ \E\left|\sum_{j=1}^n X_j\right|^p: \ n\in\mathbb{N} \mbox{ and } X\in \cU_{V}(n, p, A,B) \right\}  =  \left(\frac{B^p}{A^2}\frac{\|V\|_2^2}{\|V\|_p^p}\right)^{\frac{p}{p-2}}\E\left|\sum_{j=1}^\xi \tilde V_j\right|^p,
\end{equation}
provided that $V$ is in $L_p$, {\purp where $\tilde V_1, \tilde V_2, \dots$ are i.i.d. copies of $V$ conditioned on $\{V \neq 0\}$ ($V$ with a potential atom at $0$ \emph{removed})} and $\xi$ is an independent Poisson random variable with parameter $\left(\frac{A}{B}\frac{\|V\|_p}{\|V\|_2}\right)^{\frac{2p}{p-2}}(1-\p{V=0})$.
\end{theorem}

Observe that when $2<p<4$, the solution \eqref{eq:ros-mix-p<4} of the extremal problem is independent of the law of $V$. Theorem \ref{thm:ros-mix} readily implies the following optimal Rosenthal inequality for mixtures.

\begin{corollary}\label{cor:ros-mix}
The best constant ${\bf C}_{p,V}$ such that  for any sequence $X_1, X_2,\dots$ of independent $V$-mixtures and any $n\geq1$, we have the inequality
\[\n{\sum_{j=1}^n X_j}_p \leq {\bf C}_{p,V}\max\left\{\left(\sum_{j=1}^n \n{X_j}_2^2\right)^{1/2}, \left(\sum_{j=1}^n \n{X_j}_p^p\right)^{1/p}\right\}\]
is
\begin{equation}\label{eq:ros-Cp-mix}
{\bf C}_{p,V} = \begin{cases} (1 + \|Z\|_p^p)^{1/p}, & 2 < p \leq 4, \\ \left(\frac{\|V\|_2^2}{\|V\|_p^p}\right)^{\frac{1}{p-2}}\n{\sum_{j=1}^\xi \tilde V_j}_p, & p \geq 4, \end{cases}
\end{equation}
where $Z$ is a standard Gaussian random variable, {\purp $\tilde V_1, \tilde V_2, \dots$ are i.i.d. copies of $V$ conditioned on $\{V \neq 0\}$} and $\xi$ is an independent Poisson random variable with parameter $\left(\frac{\|V\|_p}{\|V\|_2}\right)^{\frac{2p}{p-2}}(1-\p{V=0})$.
\end{corollary}
\begin{proof}
By its definition and homogeneity, ${\bf C}_{p,V}$ satisfies
\[
{\bf C}_{p,V}^p = \sup\left\{ \E\left|\sum_{j=1}^n X_j\right|^p: \ n\in\mathbb{N} \mbox{ and } X\in \cU_{V}(n, p, 1,1) \right\},
\]
and thus \eqref{eq:ros-Cp-mix} follows immediately from \eqref{eq:ros-mix-p<4} and \eqref{eq:ros-mix-p>4} with $A = B = 1$.
\end{proof}

Theorem \ref{thm:ros-mix} specialised to Gaussian mixtures gives the sharp constants in Rosenthal's inequality for sums of nonnegative random variables. This recovers and provides a new proof of the main results obtained independently in \cite{ISh2} by Ibragimov and Sharakhmetov and in \cite{Sch} by Schechtman which use different approaches: the former adapts Utev's arguments, whereas the latter reduces the problem to sums of independent Poisson random variables using convexity and then optimises over their parameters (in an ingenious way).

\begin{corollary}[\cite{ISh2, Sch}]\label{cor:ros-pos}
Let $A, B > 0$. We have
\begin{equation}\label{eq:ros-pos}
\sup\ \E\left(\sum_{j=1}^n X_j\right)^p = \begin{cases} A^p + B^p, & 1 < p < 2, \\
\left(\frac{B^p}{A}\right)^{\frac{p}{p-1}}\E\xi^p, & p \geq 2,
\end{cases}
\end{equation}
where $\xi$ is a Poisson random variable with parameter $\left(\frac{A}{B}\right)^{\frac{p}{p-1}}$ and the supremum is taken over all $n \geq 1$ and all sequences $(X_1, \dots, X_n)$ of independent, nonnegative random variables with
\[
\sum_{j=1}^n \E X_j \leq A \qquad\text{and}\qquad \sum_{j=1}^n \E X_j^p \leq B^p.
\]
\end{corollary}
\begin{proof}
For a fixed sequence of independent positive random variables $(X_1, \dots, X_n)$, we consider the sequence $(\sqrt{X_1}Z_1, \dots, \sqrt{X_n}Z_n)$ of independent Gaussian mixtures and apply Theorem \ref{thm:ros-mix} (we note that $\sum_{j=1}^\xi Z_j$ has the same distribution as $\sqrt{\xi}Z_1$).
\end{proof}

{\purp Our main theorem also immediately yields natural complex counterparts, when symmetry is replaced with rotational invariance on the complex plane. We say that a $\C$-valued random variable $X$ is rotationally invariant (sometimes referred to as Reinhard-symmetric or circled), if $e^{it}X$ has the same distribution as $X$, for every $t \in \R$. Equivalently, $X$ has the same distribution as $|X|\eta$, where $\eta$ is a Steinhaus random variable (uniform on the unit circle $\{z \in \C, \ |z| = 1\}$), independent of $X$.  For brevity but still  illustrating the main point, we only state an analogue of \eqref{eq:ros-Cp}.

\begin{corollary}\label{cor:complex}
The best constant ${\bf C}_{p, \C}$ such that for any sequence $X_1, X_2, \dots$ of $\C$-valued independent rotationally invariant random variables and any $n \geq 1$, we have the inequality
\[
\left(\E\left|\sum_{j=1}^n X_j\right|^p\right)^{1/p} \leq {\bf C}_{p, \C} \max\left\{\left(\sum_{j=1}^n \E|X_j|^2\right)^{1/2}, \left(\sum_{j=1}^n \E|X_j|^p\right)^{1/p}\right\}
\]
reads
\[
{\bf C}_{p, \C} = \begin{cases} \left(1 + \beta_p2^{-p/2}\E|Z|^p\right)^{1/p}, & 2 < p < 4, \\
\beta_p^{1/p}\left\|\sum_{j=1}^\xi \cos(2\pi U_j)\right\|_p, & p \geq 4, \end{cases}
\]
where $Z$ is a standard Gaussian random variable, $U_1, U_2, \dots$ are i.i.d. random variables uniform on $[0,1]$, $\xi$ is an independent Poisson random variable with parameter $1$ and $\beta_p:=1/\E|\cos(2\pi U_1)|^p = \frac{\sqrt{\pi}\Gamma\left(\frac{p+2}{2}\right)}{\Gamma\left(\frac{p+1}{2}\right)}$.
\end{corollary}
\begin{proof}
Let $\eta_1, \eta_2, \dots$ be i.i.d. Steinhaus random variables, say $\eta_j = e^{2\pi i U_j}$ for i.i.d. uniform $[0,1]$ random variables $U_j$. The key is the observation that thanks to rotational invariance, complex moments are proportional to real moments of one-dimensional projections, that is
\[
\E\left|\sum_{j=1}^n X_j\right|^p = \E\left|\sum_{j=1}^n |X_j|\eta_j\right|^p = \beta_p \E\left|\sum_{j=1}^n |X_j|\cos(2\pi U_j)\right|^p
\]
This goes back to K\"onig and Kwapie\'n's paper \cite{KK} (see Lemma 8 therein). Thus we apply Theorem \ref{thm:ros-mix} to the mixtures of $\cos(2\pi U_1)$ and constraints $A = 1/\sqrt{2}$, $B = \beta_p^{-1/p}$ and the assertion follows.
\end{proof}
}


\subsection{Log-concave random variables}

An important step in Utev's approach leading to the optimal Rosenthal constant \eqref{eq:ros-Cp} for $p > 4$ is a reduction to $3$-point distributions: given $n$ and positive sequences $a, b$, it turns out that $\sup_{X \in \cM(n,p,a,b)} \E\left|\sum_{j=1}^n X_j\right|^p$ is attained at the (uniquely determined) $3$-point distribution which attains the moment constraints as equalities (for details, see Theorem \ref{thm:Utev-3point} in the next section). Our main result here, put informally, is a log-concave analogue of this. Recall that a random variable $X$ is log-concave if it has a density of the form $e^{-\phi}$ for a convex function $\phi\colon \R\to (-\infty,+\infty]$. Log-concave distributions arise naturally in convex geometry and geometric functional analysis (see e.g. \cite{AGM, BGVV}). To state our theorem rigorously, we first need to describe the relevant extremal log-concave distributions (whose convex potentials $\phi$ turn out to be piecewise linear with at most two pieces on the positive semiaxis).

For two parameters $\alpha \in [0,+\infty)$ and $\gamma \in (0,+\infty]$, we define the log-concave density
\begin{equation}
f_{\alpha,\gamma}(x) = \frac{1}{2(\alpha+1/\gamma)}\exp\left\{-\gamma(|x|-\alpha)_+\right\}
\end{equation}
with the convention that $\gamma = +\infty$, $\alpha > 0$ gives the uniform density on $[-\alpha,\alpha]$ and that $\alpha = 0$, \mbox{$\gamma < +\infty$} gives the two-sided exponential density\mbox{ $\frac{\gamma}{2}e^{-\gamma|x|}$. For $\alpha \in (0,+\infty]$ and $\gamma \in [0,+\infty)$,  let}
\begin{equation}
g_{\alpha,\gamma}(x) = \frac{\gamma}{2(1-e^{-\alpha\gamma})}\exp\left\{-\gamma|x|\right\}\1_{\{|x| \leq \alpha\}}
\end{equation}
with the convention that $\gamma = 0$, $\alpha > 0$ gives the uniform density on $[-\alpha,\alpha]$ and that $\alpha = +\infty$, $\gamma > 0$ gives the two-sided exponential density $\frac{\gamma}{2}e^{-\gamma|x|}$.
We set
\begin{equation}
\cF^- = \big\{f_{\alpha,\gamma}:\ \alpha \in [0,+\infty), \gamma \in (0,+\infty]\big\}
\end{equation}
and
\begin{equation}
\cF^+ = \big\{g_{\alpha,\gamma}:\ \alpha \in (0,+\infty], \gamma \in [0,+\infty)\big\}
\end{equation}
to be the two-parameter family of such densities. The following lemma describes the set of feasible moment parameters. We defer its simple but technical proof to {\red Section \ref{sec:proofs-logc}}.

\begin{lemma}\label{lm:uniq-f}
Let $p > 2$. For every $a, b > 0$ such that
\begin{equation}\label{eq:feas}
3^{1/2}(p+1)^{-1/p} \leq \frac{b}{a} \leq 2^{-1/2}\Gamma(p+1)^{1/p},
\end{equation}
there exist unique $f^- \in \cF^-$ and $f^+ \in \cF^+$ with 
\begin{equation}
\int_{\R} x^2f^\pm(x)\dd x = a^2 \quad \text{and} \quad \int_{\R} |x|^pf^\pm(x)\dd x = b^p.
\end{equation}
Conversely, the second and $p$-th moment of any symmetric log-concave variable satisfy \eqref{eq:feas}.
\end{lemma}

With notation set up, the main theorem of this section is the solution of a Rosenthal-type extremal problem for symmetric log-concave random variables with all moments prescribed.

\begin{theorem}\label{thm:logc-ab}
Fix $p > 4$, $n\geq1$ and let $a_1, \dots, a_n, b_1, \dots, b_n > 0$ be such that each ratio $b_j/a_j$ satisfies \eqref{eq:feas}. Then, we have
\begin{equation} \label{eq:coinc}
\inf\ \E\left|\sum_{j=1}^n X_j\right|^p = \E\left|\sum_{j=1}^n {\bf X}_j^-\right|^p \qquad \mbox{and} \qquad \sup\ \E\left|\sum_{j=1}^n X_j\right|^p = \E\left|\sum_{j=1}^n {\bf X}_j^+\right|^p,
\end{equation}
where the infimum (respectively supremum) is taken over all sequences $X = (X_1,\dots,X_n)$ of independent symmetric log-concave random variables $X_1,\ldots,X_n$ with $\E X_j^2 = a_j^2$ and $\E|X_j|^p = b_j^p$ and the ${\bf X}_j^-$ (resp.~${\bf X}_j^+$) have the unique densities from\mbox{ $\cF^-$ (resp.~$\cF^+$) that satisfy the same constraints.}
\end{theorem}


\subsection{Random variables with log-concave tails} A symmetric random variable $X$ is said to have log-concave tails if $T_X(t) = \mathbb{P}(|X|>t)$ is log-concave on $[0,\infty)$. Every log-concave random variable has log-concave tails but the converse is not true, e.g.~for Rademacher variables.  A modification of the proof of Theorem \ref{thm:logc-ab} allows us to also resolve the corresponding Rosenthal-type extremal problems for variables with log-concave tails. Consider the classes of distributions
\begin{equation}
\cG^- = \big\{X: \ T_X(t) = e^{-a(t-b)_+} \ \mbox{for some } a,b\geq0 \big\}
\end{equation}
and
\begin{equation}
\cG^+ = \big\{X: \ T_X(t) = e^{-at}{\bf 1}_{[0,b]}(t) \ \mbox{for some } a,b\geq0 \big\}.
\end{equation}

\begin{theorem}\label{thm:log-tail}
Fix $p > 4$, $n\geq1$ and let $a_1, \dots, a_n, b_1, \dots, b_n > 0$ be feasible sequences of second and $p$-th moments of symmetric random variables with log-concave tails. Then, we have
\begin{equation}
\inf\ \E\left|\sum_{j=1}^n X_j\right|^p = \E\left|\sum_{j=1}^n {\bf X}_j^-\right|^p \qquad \mbox{and} \qquad \sup\ \E\left|\sum_{j=1}^n X_j\right|^p = \E\left|\sum_{j=1}^n {\bf X}_j^+\right|^p,
\end{equation}
where the infimum and supremum are taken over all sequences $X = (X_1,\dots,X_n)$ of independent symmetric random variables $X_1,\ldots,X_n$ with log-concave tails, $\E X_j^2 = a_j^2$ and $\E|X_j|^p = b_j^p$ and the ${\bf X}_j^-$ (resp.~${\bf X}_j^+$) are the unique laws in\mbox{ $\cG^-$ (resp.~$\cG^+$) that satisfy the same constraints.}
\end{theorem}



\section{Proof of Theorem \ref{thm:ros-mix}}

We first quickly show the identity \eqref{eq:ros-mix-p<4} for $2<p<4$ as it follows from the case of symmetric random variables (Rademacher mixtures) proven by Utev in \cite{U}. 

\begin{proof}[Proof of \eqref{eq:ros-mix-p<4}]
Since $\cM_V \subseteq \cM$ and $\cU_V \subseteq \cU$, thanks to the result for symmetric random variables (see \cite[Proposition~8.1]{FHJSZ} or equation (3) in \cite{ISh}), we immediately obtain that \eqref{eq:ros-mix-p<4} holds with ``$\leq$'' in place of ``$=$''. To argue that ``$\geq$'' holds as well, it suffices to modify the example from the proof of \cite[Proposition~8.1]{FHJSZ} by replacing the Rademacher random variables with i.i.d.~copies of $V$. For completeness, we now sketch this construction. Fix $0 < \alpha < A/\n{V}_2$ and set
\begin{equation}
X_j = \begin{cases} \frac{\alpha}{\sqrt{n}}V_j, & 1 \leq j \leq n, \\ \gamma \theta_jV_j, & n < j \leq 2n, \end{cases}
\end{equation}
where $V_1, \dots, V_{2n}$ are i.i.d.~copies of $V$, $\theta_{n+1}, \dots, \theta_{2n}$ are i.i.d. Bernoulli random variables with parameter $\lambda/n$, independent of the $V_j$, and $\gamma, \lambda > 0$ are parameters to be chosen soon. For $r > 0$,
\[
\sum_{j=1}^{2n} \n{X_j}_r^r = \|V\|_r^r(\alpha^rn^{1-r/2}+\gamma^r\lambda).
\]
Thus, to ensure that $X \in \cU_V(n,p,A,B)$, it suffices to choose $\gamma$ and $\lambda$ such that
\[
\gamma^2\lambda = \frac{A^2}{\|V\|_2^2} - \alpha^2 \qquad \text{and} \qquad \gamma^p\lambda = \frac{B^p}{\|V\|_p^p} - \alpha^pn^{1-p/2}
\]
which is clearly possible as long as $n$ is large enough. Using that for arbitrary independent symmetric random variables $S, T$, we have $\E|S+T|^p \geq \E|S|^p + \E|T|^p$, $p > 2$ (see \eqref{eq:ros-easy}, say), we get
\[
\E\left|\sum_{j=1}^{2n} X_j\right|^p \geq \E\left|\sum_{j=1}^{n} X_j\right|^p + \sum_{j=n+1}^{2n} \E|X_j|^p = \alpha^p\E\left|\frac{\sum_{j=1}^nV_j}{\sqrt{n}}\right|^p + \gamma^p\lambda\|V\|_p^p.
\]
The family $\left\{\left|\frac{\sum_{j=1}^nV_j}{\sqrt{n}}\right|^p\right\}_{n \geq 1}$ is uniformly integrable (since it is bounded in $L_{1+\delta/p}$, by \eqref{eq:ros}, say), so by the central limit theorem, $\E\left|\frac{\sum_{j=1}^nV_j}{\sqrt{n}}\right|^p \to \|Z\|_p^p\|V\|_2^p$ as $n \to \infty$. Moreover, $\gamma^p\lambda\|V\|_p^p = B^p - \alpha^pn^{1-p/2}\|V\|_p^p \to B^p$ as $n\to\infty$ and letting $\alpha \to A/\|V\|_2$ finishes the argument.
\end{proof}

The proof of \eqref{eq:ros-mix-p>4} requires some preparation. We first recall some of Utev's results. Central to his approach is the following Poissonisation estimate. For a finite nonnegative Borel measure $\nu$ on $\R$, we denote by $T_\nu$ a random variable with characteristic function 
\begin{equation}
\E e^{itT_\nu} = \exp\left\{\int_{\R}(e^{itx}-1)\dd\nu(x)\right\}, \qquad t \in \R.
\end{equation}
($T_\nu$ can be explicitly constructed as $\sum_{j=1}^\xi X_j$, where $X_1, X_2, \dots$ are i.i.d. copies of a random variable with law $\frac{1}{\nu(\R)}\nu$ and $\xi$ is an independent Poisson random variable with parameter $\nu(\R)$.) As customary, $\cB(\R)$ denotes the $\sigma$-algebra of all Borel sets in $\R$.

\begin{theorem}[Utev, Theorem 4 in \cite{U}]\label{thm:Utev-Poisson}
Let $\Phi\colon \R \to \R$ be an even $C^2$ function with $\Phi''$ convex. For every $n\in\mathbb{N}$ and independent symmetric random variables $X_1, \dots, X_n$, we have
\begin{equation}\label{eq:Utev-Poisson}
\E\Phi\left(\sum_{j=1}^n X_j\right) \leq \E\Phi(T_\nu)
\end{equation}
with $\nu$ defined by setting
\begin{equation}
\nu(\Gamma) = \sum_{j=1}^n \p{X_j \in \Gamma \setminus \{0\}}, \qquad \Gamma \in \cB(\R).
\end{equation}
\end{theorem}

We will also need to use the fact that $3$-point distributions are extremal among all symmetric distributions with fixed moments.

\begin{theorem}[Utev, Theorem 5 in \cite{U}]\label{thm:Utev-3point}
For $p \geq 4$, we have
\begin{equation}\label{eq:Utev-M}
\sup_{X \in \cM(n,p,a,b)} \E\left|\sum_{j=1}^n X_j\right|^p = \E\left|\sum_{j=1}^n \left(\frac{b_j^p}{a_j^2}\right)^{\frac{1}{p-2}}\theta_{j,\mu_j}\e_j\right|^p,
\end{equation}
where the $\theta_{j,\mu_j}$ denote i.i.d. Bernoulli random variables with parameter $\mu_j = \left(a_j/b_j\right)^{\frac{2p}{p-2}}$, {\red independent of the sequence $(\e_j)$}.
\end{theorem}

We are ready to prove \eqref{eq:ros-mix-p>4}. We begin with a lemma showing that i.i.d.~sequences are extremal.

\begin{lemma}\label{lm:to-i.i.d.}
Let $p \geq 3$, $A, B > 0$ and let $V$ be a symmetric random variable in $L_p$. We have,
\begin{equation} \label{eq:iiiid}
\sup_{n \geq 1,\ X \in \cU_{V}(n, p, A,B)} \E\left|\sum_{j=1}^n X_j\right|^p = \sup_{n \geq 1,\ X \in \cU_{V}'(n, p, A,B)} \E\left|\sum_{j=1}^n X_j\right|^p.
\end{equation}
\end{lemma}
\begin{proof}
We shall say that a finite nonnegative Borel measure $\nu$ on $\R$ is a {\red $V$-measure} if it is of the form
\[
\nu(\Gamma) = \int_0^\infty \p{rV \in \Gamma \setminus \{0\}} \dd \eta(r), \qquad \Gamma \in \cB(\R)
\]
for some finite nonnegative Borel measure $\eta$ on $(0,\infty)$. We shall argue that both suprema in \eqref{eq:iiiid} are equal to the proxy
\begin{equation}
Q = \sup_{\nu \in \cW_V(p,A,B)} \E|T_\nu|^p,
\end{equation}
where
\[
\cW_V(p,A,B) = \left\{\nu: \ \nu\mbox{ is a }V\mbox{-mixture with} \ \nu(\{0\}) = 0, \ \int_{\R} x^2 \dd \nu(x) \leq A^2, \   \int_{\R} |x|^p \dd \nu(x) \leq B^p \right\}.
\]

\noindent \emph{Step I.} We have,
\[
\sup_{n \geq 1, X \in \cU_{V}'(n, p, A,B)} \E\left|\sum_{j=1}^n X_j\right|^p \leq \sup_{n \geq 1, X \in \cU_{V}(n, p, A,B)} \E\left|\sum_{j=1}^n X_j\right|^p \leq Q,
\]
where the first inequality is clear and the second one follows from \eqref{eq:Utev-Poisson} for $\Phi(x)=|x|^p$. 

\noindent \emph{Step II.}
We now have to show that
\[
Q \leq \sup_{n \geq 1, X \in \cU_{V}'(n, p, A,B)} \E\left|\sum_{j=1}^n X_j\right|^p.
\]
To this end, fix $\nu \in \cW_V(p,A,B)$. Take an integer $n \geq \nu(\R)$ and i.i.d.~random variables $X_1, \dots, X_n$ with the law specified by
\[
\p{X_1 \in \Gamma} = \frac{1}{n}\nu(\Gamma), \qquad \Gamma \in \cB(\R)\mbox{ with } 0 \notin \Gamma
\]
and $\p{X_1 = 0} = 1 - \frac{1}{n}\nu(\R)$. Since $\nu$ is a {\red $V$-measure}, the random variables $X_j$ are also i.i.d.~$V$-mixtures. Moreover, $X = (X_1, \dots, X_n) \in \cU'_V(n,p,A,B)$, as $\nu \in \cW_V(p,A,B)$. Note that the characteristic function of $\sum_{j=1}^n X_j$ equals
\[
\left(1 - \frac{1}{n}\nu(\R) + \frac{1}{n}\int_{\R} e^{itx} \dd\nu(x)\right)^n = \left(1 + \frac{1}{n}\int_{\R} (e^{itx}-1) \dd\nu(x)\right)^n,
\]
and thus $\sum_{j=1}^n X_j$ converges in distribution to $T_\nu$ as $n \to \infty$. By Fatou's lemma, we get
\[
\E|T_{\nu}|^p \leq \liminf_{n \to \infty} \E\left|\sum_{j=1}^n X_j\right|^p
\]
which finishes the proof of \eqref{eq:iiiid}.
\end{proof}

\begin{remark}
We stress here that the above proof of Lemma \ref{lm:to-i.i.d.} works for $p\geq 3$. It is not known to us if the same assertion, i.e. that the i.i.d. and non-i.i.d. Rosenthal constants in the class of $V$-mixtures match, remains true in the regime $2\leq p\leq 3$. Note that our example in the proof of \eqref{eq:ros-mix-p<4}, which works for $2\leq p\leq 4$, consists of two blocks of i.i.d. rather than a single sequence of i.i.d. random variables.
\end{remark}

\begin{proof}[Proof of \eqref{eq:ros-mix-p>4}]
By virtue of Lemma \ref{lm:to-i.i.d.}, we will equivalently show \eqref{eq:ros-mix-p>4} with $\cU_V'$ in place of $\cU_V$. First we argue that ``$\leq$'' holds. To this end, take an $n$-tuple $X = (X_1, \dots, X_n)$ of i.i.d.~$V$-mixtures in $\cU_V(n,p,A,B)$, say $X_j = R_j|V_j|\e_j$ for some i.i.d.~nonnegative random variables $R_j$ and such that the $R_j, V_j, \e_j$ are all independent. Plainly,
\[
\E\left|\sum_{j=1}^n X_j\right|^p = \E_V\E_{R,\e}\left|\sum_{j=1}^n |V_j|R_j\e_j\right|^p,
\]
so conditioning on the values of the $V_j$ and applying \eqref{eq:Utev-M} yields
\begin{align*}
\E_V\E_{R,\e}\left|\sum_{j=1}^n |V_j|R_j\e_j\right|^p &\leq \E_V\E_{\theta,\e}\left|\sum_{j=1}^n \left(\frac{(Bn^{-1/p})^p}{(An^{-1/2})^2}\frac{\|V\|_2^2}{\|V\|_p^p}\right)^{\frac{1}{p-2}}|V_j|\theta_{j}\e_j\right|^p \\
&=\left(\frac{B^p}{A^2}\frac{\|V\|_2^2}{\|V\|_p^p}\right)^{\frac{p}{p-2}}\E\left|\sum_{j=1}^n \theta_j V_j\right|^p,
\end{align*}
where the $\theta_j$ are i.i.d. Bernoulli random variables with parameter 
\[
\mu = \left(\frac{An^{-1/2}}{Bn^{-1/p}}\frac{\|V\|_p}{\|V\|_2}\right)^{\frac{2p}{p-2}} = n^{-1} \left(\frac{A}{B}\frac{\|V\|_p}{\|V\|_2}\right)^{\frac{2p}{p-2}}.
\]
Let $\nu$ be a Borel measure specified by
\begin{equation}\label{eq:nu}
\nu(\Gamma) = \left(\frac{A}{B}\frac{\|V\|_p}{\|V\|_2}\right)^{\frac{2p}{p-2}}\p{V \in \Gamma\setminus\{0\}} = n\p{\theta_1V_1 \in \Gamma\setminus\{0\}}, \qquad \Gamma \in \cB(\R).
\end{equation}
Then Theorem \ref{thm:Utev-Poisson} yields
\[
E\left|\sum_{j=1}^n \theta_j V_j\right|^p \leq \E|T_\nu|^p.
\]
Note that $T_\nu$ has the same distribution as $\sum_{j=1}^\xi \tilde V_j$, where $\xi$ is a Poisson random variable with parameter $\nu(\R)$ {\purp and $\tilde V_1, \tilde V_2, \dots$ are i.i.d. copies of a random variable with distribution $\tilde\nu(\Gamma) = \frac{1}{\nu(\R)}\nu(\Gamma) = \frac{\p{V \in \Gamma \setminus \{0\}}}{1-\p{V = 0}}$}, as in the statement of the theorem. This finishes the proof of the inequality ``$\leq$'' in \eqref{eq:ros-mix-p>4}. The reverse inequality follows by repeating the construction from Step II of the proof of Lemma \ref{lm:to-i.i.d.}. Alternatively, recalling from the proof of Lemma \ref{lm:to-i.i.d.}  the notation of the class $\cW_V$ as well as the proxy quantity $Q$, we take the measure $\nu$ defined in \eqref{eq:nu} above and rescale
\[
\tilde\nu(\Gamma) = \nu\left(\Gamma\Big/\left(\frac{B^p}{A^2}\frac{\|V\|_2^2}{\|V\|_p^p}\right)^{\frac{1}{p-2}}\right), \qquad \Gamma \in \cB(\R)
\]
to have, $\tilde\nu \in \cW_V(p,A,B)$, thus by the definition of $Q$,
\[
\left(\frac{B^p}{A^2}\frac{\|V\|_2^2}{\|V\|_p^p}\right)^{\frac{p}{p-2}} \E|T_\nu|^p = \E|T_{\tilde \nu}|^p \leq Q,
\]
which finishes the proof in view of two facts, that $T_\nu$ has the same distribution as $\sum_{j=1}^\xi \tilde V_j$ and that $Q$ equals the supremum from \eqref{eq:ros-mix-p>4} (as shown in the proof of Lemma \ref{lm:to-i.i.d.}).
\end{proof}

{\red
\begin{remark}\label{rem:Mp>4}
Under the assumptions of Theorem \ref{thm:ros-mix}, for $p \geq 4$, we also have a result for individually constrained summands, viz.
\[
 \sup_{X\in \cM_{V}(n, p, a,b)} \E\left|\sum_{j=1}^n X_j\right|^p = \E\left|\sum_{j=1}^n \left(\frac{b_j^p}{a_j^2}\frac{\|V\|_2^2}{\|V\|_p^p}\right)^{\frac{1}{p-2}}\theta_{j,\mu_j}V_j\right|^p
\]
with $\mu_{j} = \left(\frac{a_j}{b_j}\frac{\|V\|_p}{\|V\|_2}\right)^{2p/(p-2)}$. This immediately follows from Utev's Theorem \ref{thm:Utev-3point} because conditioning on the values of the $V_j$ and applying \eqref{eq:Utev-M} (with budget sequences $a_j\frac{|V_j|}{\|V\|_2}$ and $b_j\frac{|V_j|}{\|V\|_p}$) yields the upper bound
\[
\E\left|\sum_{j=1}^n X_j\right|^p \leq \E\left|\left(\frac{b_j^p}{a_j^2}\frac{\|V\|_2^2}{\|V\|_p^p}\right)^{\frac{1}{p-2}}\theta_{j,\mu_j}V_j\right|^p
\]
which is attained as every summand on the right hand side is a $V$-mixture satisfying the moment constraints.
\end{remark}
}


\section{Proof of Theorems \ref{thm:logc-ab} and \ref{thm:log-tail}}\label{sec:proofs-logc}

\subsection{Log-concave random variables} We only consider the case of the infimum and class $\cF^-$, with the obvious modifications left out to address the case of the supremum and $\cF^+$.

\begin{proof}[Proof of Lemma \ref{lm:uniq-f}]
Let $f_0, f_1 \in \cF^-$ be the uniform and two-sided exponential densities chosen such that $\int_\R x^2 f_j(x) \dd x = a^2$, $j = 0,1$. Then
\[
\int_\R |x|^pf_0(x) \dd x = 3^{p/2}(p+1)^{-1}a^p, \qquad \int_\R |x|^pf_1(x) \dd x = 2^{-p/2}\Gamma(p+1)a^p
\]
For $\rho \in [0,+\infty]$, let
\[
\gamma = \gamma(\rho) = a^{-1}\sqrt{2+\frac{\rho^3+3\rho^2}{3(\rho+1)}}
\]
which is chosen such that the density $g_\rho(x) =  f_{\frac{\rho}{\gamma}, \gamma}$ satisfies
\[
\int_\R x^2g_\rho(x) \dd x = a^2.
\]
Since $g_0 = f_1$ and $g_\infty = f_0$, the intermediate value property gives the existence of $\rho$ such that $\int_\R |x|^pg_\rho(x) \dd x = b^p$, as desired. The uniqueness follows from the fact that {\red for two arbitrary distinct densities from $\cF^-$, their difference changes sign at most twice on $(0,\infty)$}. The converse implication is classical, see for instance \cite[Remark~14]{ENT} and \cite[Proposition~5.5]{GNT}.
\end{proof}

We need several more ancillary results.

\begin{lemma}\label{lm:f-f0}
Let $g$ be an even log-concave density on $\R$ and let $f \in \cF^-$. Then $g-f$ changes sign at most $3$ times on $(0,+\infty)$.
\end{lemma}
\begin{proof}
Say $f = f_{\alpha,\gamma}$. Examining $\log g - \log f$ on $(0,+\infty)$, this difference clearly changes sign at most once on $(0,\alpha]$ (by monotonicity of $\log g$, since $\log f$ is constant there) and at most twice on $[\alpha,+\infty)$ (by concavity of $\log g$, since $\log f$ is linear there).
\end{proof}

\begin{lemma}\label{lm:psi_p}
Let $p > 4$. The following function
\begin{equation}
\psi_p(x) = |\sqrt{x}+1|^p + |\sqrt{x}-1|^p - 2x^{p/2}, \qquad x \geq 0
\end{equation}
is strictly convex on $(0,+\infty)$.
\end{lemma}
\begin{proof}
Let $g(x) = |x+1|^p+|x-1|^p-2x^p$, $x \geq 0$. We have $\psi_p(x) = g(\sqrt{x})$ and $\psi_p'(x) = \frac{g'(\sqrt{x})}{2\sqrt{x}}$. Since $g'(0) = 0$, it suffices to show that $g'$ is convex, because then $\psi_p'$ is increasing. For $x > 1$, we have
\[
\frac{1}{p(p-1)(p-2)}g'''(x) = (x+1)^{p-3}+(x-1)^{p-3} - 2x^{p-3} > 0,
\]
by {\red the convexity of $x \mapsto x^{p-3}$}. For $0 < x  <1$, we have
\begin{align*}
\frac{1}{p(p-1)(p-2)(p-3)}g'''(x) &= \frac{(x+1)^{p-3}-(1-x)^{p-3} - 2x^{p-3}}{p-3} \\
&= \int_{1-x}^{1+x} t^{p-4} \dd t - 2\int_0^x t^{p-4}\dd t > 0,
\end{align*}
by the monotonicity of the integrand.
\end{proof}

\begin{lemma}\label{lm:det}
For $0 < x_1 < x_2 < x_3$ and a convex function $\phi\colon (0,+\infty) \to \R$, we have
\begin{equation}
\det\begin{bmatrix} 1 & x_1 & \phi(x_1) \\ 1 & x_2 & \phi(x_2) \\ 1 & x_3 & \phi(x_3) \end{bmatrix} \geq 0.
\end{equation}
Moreover, the inequality is strict if $\phi$ is strictly convex.
\end{lemma}

\begin{proof}
The desired inequality is equivalent to
\[
\frac{x_3-x_2}{x_3-x_1}\phi(x_1) + \frac{x_2-x_1}{x_3-x_1}\phi(x_3) \geq \phi(x_2)
\]
which clearly follows from convexity.
\end{proof}

\begin{lemma}\label{lm:sign}
Let $p > 4$, let $\alpha, \beta, \gamma$ be real numbers. The function
\[
h(x) = |x+1|^p + |x-1|^p - \alpha - \beta x^2 - \gamma x^p
\]
has at most $3$ {\red roots on $(0,+\infty)$. Moreover, if it has exactly $3$, then it changes sign at each of them and its signature is $+,-,+,-$.}
\end{lemma}
\begin{proof}
It suffices to consider the case when $h$ has $3$ distinct sign changes, say at $0 < x_1 < x_2 < x_3 < \infty$ and show that it does not have any more. First note that then, 
{\red solving the linear system of equations $h(x_j) = 0$, $j=1,2,3$ in $\gamma$},
\[
\gamma-2 = \frac{\det\left[\begin{smallmatrix} 1 & x_1^2 & |x_1+1|^p + |x_1-1|^p \\ 1 & x_2^2 & |x_2+1|^p + |x_2-1|^p \\ 1 & x_3^2 & |x_3+1|^p + |x_3-1|^p\end{smallmatrix}\right]}{\det\left[\begin{smallmatrix} 1 & x_1^2 & x_1^p \\ 1 & x_2^2 & x_2^p \\ 1 & x_3^2 & x_3^p \end{smallmatrix}\right]}-2 = \frac{\det\left[\begin{smallmatrix} 1 & x_1^2 & \psi_p(x_1^2) \\ 1 & x_2^2 & \psi_p(x_2^2) \\ 1 & x_3^2 & \psi_p(x_3^2) \end{smallmatrix}\right]}{\det\left[\begin{smallmatrix} 1 & x_1^2 & x_1^p \\ 1 & x_2^2 & x_2^p \\ 1 & x_3^2 & x_3^p \end{smallmatrix}\right]},
\]
where we use the linearity of the determinant with respect to the last column 
and adapt the notation from Lemma \ref{lm:psi_p} for $\psi_p$. It follows from Lemmas \ref{lm:psi_p} and \ref{lm:det} that both the numerator and the denominator are positive and thus $\gamma > 2$. In particular, this gives that $h^{(k)}(+\infty) = -\infty$, for $k = 0,1,2,3$.
Furthermore,
\[
h'''(x) = p(p-1)(p-2)x^{p-3}\big(|1+x^{-1}|^{p-3}+|1-x^{-1}|^{p-3}\sgn(1-x^{-1})-\gamma\big)
\]
and we verify that the function in the brackets {\red has exactly one root on $(0, +\infty)$ at which it changes sign from $+$ to $-$. As $h''(+\infty) = -\infty$, we gather that $h''$  has either exactly one root on $(0,+\infty)$ and signature $+,-$, in which case the conclusion follows, or exactly two, with the signature $-,+,-$. In the latter case, since $h'(0) = 0$ and $h'(+\infty) = -\infty$, we get at most two roots of $h'$ on $(0,+\infty)$, hence $h$ has at most $3$ roots on $(0,+\infty)$. If it has exactly $3$ roots, then the signature of $h$ is $+,-,+,-$, since $h(+\infty) = -\infty$.}
\end{proof}

\begin{proof}[Proof of Theorem \ref{thm:logc-ab}]
Fix a sequence $X=(X_1,\dots,X_n)$ as in the statement, an index $1 \leq j \leq n$ and a parameter $z \in\R$. Thanks to independence, it suffices to show that
\begin{equation} \label{eq:final}
\E|X_j + z|^p \geq \E|{\bf X}_j^-+z|^p.
\end{equation}
Indeed, if \eqref{eq:final} holds, then swapping one variable at a time, we get
\[
\mathbb{E}\left| \sum_{j=1}^n X_j\right|^p \geq \mathbb{E} \left| {\bf X}_1^- + \sum_{j=2}^n X_j \right|^p \geq \cdots \geq \mathbb{E} \left| \sum_{j=1}^n{\bf X}_j^-\right|^p,
\]
which is exactly \eqref{eq:coinc}. To prove \eqref{eq:final}, let $f_j$ be the density of $X_j$ and ${\bf f}_j^-$ be the density of ${\bf X}_j^-$. Using symmetry, the inequality in question is equivalent to 
\[
\int_0^\infty \big(f_j(x) - {\bf f}_j^-(x)\big)\cdot \phi(x) \ \dd x \geq 0,
\]
where $\phi(x) = |x+z|^p+|x-z|^p$. To prove it,  we employ the technique of {\red ``interlacing densities'', \cite{ENT}}. Since the integrals of $f_j$ and ${\bf f}_j^-$ against $1$, $x^2$ and $x^p$ are the same, for every $\alpha, \beta, \gamma \in \R$, the inequality above is equivalent to
 \[
\int_0^\infty \big(f_j(x) - {\bf f}_j^-(x)\big)\cdot \big(\phi(x) - \alpha - \beta x^2 - \gamma x^p\big) \ \dd x \geq 0.
\]
For the same reason, $f_j-{\bf f}_j^-$ need change sign at least $3$ times on $(0,+\infty)$ {\red (see Lemma 21 in \cite{ENT})}. In view of Lemma \ref{lm:f-f0}, there are exactly $3$ sign changes, say at $x_1 < x_2 < x_3$ and the signature must necessarily be $+,-,+,-$ (because the first sign change will occur on the interval where ${\bf f}_j^-$ is constant and $f_j$ decreases on $(0,+\infty)$). We choose $\alpha, \beta, \gamma$ such that $f(x)=\phi(x) - \alpha - \beta x^2 - \gamma x^p$ vanishes at $x_1, x_2, x_3$. By Lemma \ref{lm:sign}, {\red $f$ in fact} changes sign on $(0,+\infty)$ exactly at those roots and has signature $+,-,+,-$. Thus the integrand is pointwise nonnegative and \eqref{eq:final} follows.
\end{proof}

{\red
\begin{remark}\label{rem:Xpm}
Given the $2$nd and $p$th moment constraints, it follows from Theorem 8 in \cite{ENT} that variables ${\bf X}_j^-$and ${\bf X}_j^+$ are also the minimisers and maximisers respectively of $\E |X_j|^q$ for every $q \in (0,2) \cup (p, \infty)$ (and maximisers and minimisers repsectively for $q \in (2, p)$). 
\end{remark}
}


\subsection{Log-concave tails}
The proof of Theorem \ref{thm:logc-ab} with minor modifications also gives Theorem \ref{thm:log-tail}. Skipping most of the details, the main point is that now we write
\begin{align*}
\E|X_j + z|^p - \E|{\bf X}_j^-+z|^p &= \int_0^\infty \phi'(t)\Big(T_{X_j}(t) - T_{{\bf X}_j^-}(t)\Big) \ \dd t \\
&= \int_0^\infty \Big(\phi'(t)-\beta t - \gamma t^{p-1}\Big)\cdot\Big(T_{X_j}(t) - T_{{\bf X}_j^-}(t)\Big) \ \dd t,
\end{align*}
where the second equality uses the constraints on the second and $p$th moments. Since ${\bf X}_j^-\in\cG^-$ (see also the definition of the classes $\mathcal{L}_2^{\pm}$ in \cite{ENT}), the difference $T_{X_j} - T_{{\bf X}_j^-}$ changes sign exactly $2$ times on $(0,+\infty)$ with signature $-,+,-$. On the other hand, choosing $\beta$ and $\gamma$ such that $\phi'(t)-\beta t - \gamma t^{p-1}$ vanishes at those sign change points, the proof of Lemma \ref{lm:sign} yields that this bracket has exactly these sign changes with signature $-,+,-$ and we arrive at the conclusion of Theorem \ref{thm:log-tail}. To get the reverse inequality (the supremum instead of the infimum), we choose ${\bf X}_j^+\in\cG^+$ and repeat the argument while noting the different sign patterns. \hfill$\Box$


\section{Further directions}


{\red
\subsection{More constraints}
We find it an interesting question how to extend the arguments of the previous section to handle more than two moment constraints, that is given $n$, say $\ell$ constraints, distinct positive parameters $p_1, \dots, p_n$ and feasible moment budgets $a_{j,k} > 0$, $j = 1, \dots, n$, $k = 1, \dots, \ell$, we would like to describe the extremisers of the infimum (respectively supremum) of $\E|\sum_{j=1}^n X_j|^p$ taken over all sequences $(X_1, \dots, X_n)$ of length $n$  consisting of independent symmetric random variables with log-concave tails/densities satisfying $\E |X_j|^{p_1} = a_{j,1}, \dots, \E |X_j|^{p_\ell} = a_{j,\ell}$, $j = 1, \dots, n$. For the case of log-concave tails, the set of feasible parameters $a_{j,k}$ is described precisely in \cite[Theorem~7]{ENT}. There are further technical challenges, for instance how to extend Lemma \ref{lm:sign}. 
}

\subsection{Sharp constants}
It has been elusive to us how to obtain an analogue of Theorem \ref{thm:ros-mix} for symmetric log-concave random variables. We find it an interesting and challenging problem. The natural approach, via Theorem \ref{thm:logc-ab}, leaves us with an optimisation problem over the $a_j$ and $b_j$, unclear how to tackle.



\end{document}